\documentclass[12pt]{article}

\usepackage{amssymb,amsthm,amsmath}
\usepackage{enumerate,bm,mathrsfs}
\usepackage{tabularx}
\usepackage{tikz,graphicx}
\usepackage{caption, subcaption}
\usepackage{tikz-3dplot,pgfplots}
\usepackage{geometry}
\usepackage{cleveref}

\pgfplotsset{compat=1.10}
\newtheorem{theorem}{Theorem}[section]
\newtheorem{corollary}[theorem]{Corollary}
\newtheorem{lemma}[theorem]{Lemma}
\newtheorem{proposition}[theorem]{Proposition}
\newtheorem{conjecture}[theorem]{Conjecture}
\newtheorem{assumption}[theorem]{Assumption}

\theoremstyle{definition}
\newtheorem{example}[theorem]{Example}
\newtheorem{definition}[theorem]{Definition}

\newtheorem{customlemma}{Lemma}
\newenvironment{customlem}[1]
  {\renewcommand\thecustomlemma{#1}\customlemma}
  {\endcustomlemma}

\providecommand{\abs}[1]{\left\lvert#1\right\rvert}
\providecommand{\norm}[1]{\left\rVert#1\right\rVert}
\providecommand{\set}[1]{ \left\{ #1  \right\}  }
\providecommand{\setb}[2]{ \left\{ #1 \ \middle| \  #2 \right\}  }
\providecommand{\innprod}[2]{\left\langle #1, #2 \right\rangle}
\providecommand{\parenth}[1]{\left( #1 \right) }

\title{Interpolating splines on graphs for data science applications\thanks{This research was supported by grants 
DMS-1514789 and DMS-1813091 from the National Science Foundation.}}

\author{J. P. Ward\thanks{North Carolina A\&T State University, Greensboro, NC.} %(\email{jpward@ncat.edu})}
\and
F. J. Narcowich\thanks{Texas A\&M University, College Station, TX.}
\and
J. D. Ward\footnotemark[3]}

\begin{document}

\maketitle

\begin{abstract}
We introduce intrinsic interpolatory bases for data structured on graphs and derive properties of those bases. Polyharmonic Lagrange functions are shown to satisfy exponential decay away from their centers. The decay depends on the density of the zeros  of the Lagrange function, showing that they \textit{scale} with the density of the data. These results indicate that Lagrange-type bases are ideal building blocks for analyzing data on graphs,
and we illustrate their use in kernel-based machine learning applications.
\end{abstract}

%\begin{keywords}
%\noindent
%Keywords:  Local basis functions, Kernel-based machine learning, Interpolation, Lagrange functions \\
%AMS: 41A05, 41A15, 41A65
%\end{keywords}

\section{Introduction}

Graph, or network, domains are being used for many signal processing applications \cite{shuman13}. They provide a more general framework than integer lattices, and they can be used to incorporate additional structural or geometric information. 

Our purpose here is to develop and analyze intrinsic interpolatory and near-interpolatory bases for graphs, with the goal of introducing more approximation tools for graphs.  
Kernels on graphs were introduced in \cite{smola03} within the context of regularization operators. 
Both Gaussian and regularized Laplacian kernels were considered. Then, interpolants on graphs were considered in \cite{pesenson09}, where the author defines \emph{variational splines} and \emph{Lagrangian splines} that are similar to the interpolants that we propose. 
There, a perturbation factor is used to modify the Laplacian into a positive definite matrix, making the basis functions analogs of the Mat\'ern functions. 
Our approach is to work directly with the Laplacian, making our splines analogs of the polyharmonic splines. 
Another distinction is that we work on weighted graphs. 
Properties of splines are also discussed in \cite{pesenson10r}. In particular, the author considers  approximation in Paley-Wiener spaces. 
Note that while we are considering Lagrange functions, our approach (based on radial basis function theory) is different than that of \cite{pesenson09,pesenson10r,pesenson08}, and hence our notation is also different. Also, our main focus is in this paper is the decay/localization of these bases, which was not previously addressed.

Our main result is an estimate on the decay of Lagrange functions as this is a key first step toward related computationally efficient bases on graphs.
Our reference point is the theory of Lagrange bases on Euclidean domains and Riemannian manifolds. Lagrangian variational splines on general Riemannian manifolds were introduced in \cite{pesenson00}. Decay and localization properties were established in \cite{hangelbroek12plc,hangelbroek10}.   On $\mathbb{R}^d$, for example, it is known that polyharmonic interpolating splines satisfy fast decay rates and form stable bases of $L_p$ spaces, among other desirable properties of approximants. Moreover, on those continuous domains, it is known that more computationally efficient perturbations of the Lagrange functions (local Lagrange functions) satisfy similar properties. We believe all of these results should carry over to the graph setting. 

While we use the continuous domain theory as a theoretical guideline, we expect the continuous and discrete theories to be complementary. For example, there are empirically estimated constants in the continuous domain results that should be more accessible for graph domains.   One such constant determines how
many basis functions are needed to construct a \textit{good} local Lagrange function. If these constants could be worked out precisely for graphs, it could shed light on the situation in other settings.

Unlike the continuous domain, polyharmonic functions on graphs are not only bounded but also exhibit fast decay \cite{sun07}, so the decay of the Lagrange functions could be deduced from this. 
In \cite{sun07},  the author deals with classes of matrices in weighted spaces, and he shows that pseudo-inverses of such matrices live in the same class.
Our goal is to obtain more precise estimates of decay that depend on the density of the data, showing that the Lagrange functions are actually much more localized than the original basis functions.

There are numerous potential applications of interpolants and quasi-interpolants on graphs. In \cref{sec:example} we discuss one of these, namely kernel-based machine learning.  Regression and classification problems based on many parameters seem to be a natural application for a graph model. The space of parameters likely does not live in a natural continuous ambient space. Using graphs, we can define the relationships that make sense and specify precisely how connected two objects (vertices) are by a weighted edge. We present an algorithm for such problems and apply it to sample data sets.

The remainder of the paper is organized as follows.  We conclude this section with some notation and background information about functions on graphs. In \cref{sec:basis_functions}, we introduce the polyharmonic basis functions and verify some basic properties. \cref{sec:lagrange} contains our main results: the construction and properties of polyharmonic Lagrange functions on graphs. 
We define local Lagrange functions in \cref{sec:loc_lag}.
In \cref{sec:example}, we illustrate our results with some examples, 
and, in \cref{sec:discussion}, we summarize our results and discuss directions for future research.

\subsection{Setting}
\label{sec:setting}

The general setting is a finite, connected, weighted graph 
$\mathcal{G} 
= 
\set{
\mathcal{V},
\mathbf{E}\subset \mathcal{V} \times \mathcal{V}, 
w,
\rho
}$ 
where $\mathcal{V}$ is the vertex set,  
$\mathbf{E}$ is the edge set, 
and $w: \mathcal{V} \times \mathcal{V} \rightarrow \mathbb{R}_{\geq 0}$ is a symmetric weight function. The weight between two vertices is zero if and only if there is no edge connecting them.

We assume $\mathcal{G}$ has at least two vertices.
We denote the maximum degree of the vertices in the graph by $M \in \mathbb{Z}_{\geq 1}$.
The weight function $w$ specifies the adjacency matrix $A$ of the graph.  
The $j$th diagonal entry of the diagonal matrix $D$ is equal to the $j$th row sum of $A$.  

The function $\rho: \mathcal{V} \times \mathcal{V} \rightarrow \mathbb{R}_{\geq 0}$ is a distance function on the graph. We require $\rho$ to be positive on distinct pairs of vertices. 
We denote the maximum distance between adjacent vertices as $\rho_{\text{max}}$.
Given the distance between adjacent vertices, the distance between non-adjacent vertices is the length of the shortest path connecting them.  

For interpolation problems where data is known on a subset of the vertices, we assume this set is not empty and we denote it as $\widetilde{\mathcal{V}}$. 

The normalized graph Laplacian is $L = D^{-1/2}(D-A)D^{-1/2}$ \cite{benedetto15,chung97}. 
Let $\set{\Lambda_k}_{k=0}^{N-1}$ and $\set{\lambda_k}_{k=0}^{N-1}$ denote the eigenvectors and eigenvalues of the Laplacian. The Laplacian is a positive semi-definite matrix, with a zero eigenvalue of multiplicity one (since the graph is connected). We order the eigenvalues as
\begin{equation}
	0
	=
	\lambda_0
	<
	\lambda_1
	\leq
	\cdots
	\leq
	\lambda_{N-1}.
\end{equation} 
The eigenvector $\Lambda_0$ is a multiple of $D^{1/2}e$, where $e$ is a constant vector.

The eigenvalues of $L$ are bounded above by $2$. This follows from Theorem 2.7 of \cite{grigor18}.

\begin{example}
Let $\mathcal{G}$ be the cycle graph with $N \in \mathbb{Z}_{\geq 1}$ nodes and edges of equal weight and distance $1$. In this case the Laplacian is the circulant matrix with rows $\dots,0, 0,-1/2,1,-1/2,0,0,\dots$.
\end{example}

On a graph $\mathcal{G}$, we denote the closed ball of radius $r$ centered at $v$ as $B(v;r)$. The annulus centered at $v$ with inner and outer radii $r_0$ and $r_1$ respectively is denoted as $A(v;r_0,r_1)$. The annulus is defined by the set difference $B(v;r_1)\backslash B(v;r_0)$.  The radius $r$ is computed with respect to the distance function $\rho$.

\subsection{Background: Smoothness spaces}

A function on a graph is a mapping 
\begin{equation}
	f: \mathcal{V} \rightarrow \mathbb{R}.
\end{equation}
The space of square-summable functions on any subset $\mathcal{W} \subset \mathcal{V}$ is denoted 
$\ell_2(\mathcal{W})$. In some instances it is advantageous to specify these spaces in terms of the graph, so we also use the notation the notation $\ell_2(\mathcal{G})$ to represent functions on the vertex set of 
$\mathcal{G}$.

We shall primarily focus on (semi-)Hilbert spaces on graphs, and we measure smoothness using the Laplacian. For any $\alpha>0$, we define the Sobolev semi-norm 
\begin{equation}
	\abs{f}_{H_{2}^{\alpha}(\mathcal{G})}
	=
	\norm{L^{\alpha/2}f}_{\ell_2(\mathcal{G})}.
\end{equation}
For a subset $\mathcal{W} \subset \mathcal{V}$,
\begin{equation}
	\abs{f}_{H_{2}^{\alpha}(\mathcal{W})}
	=
	\norm{L^{\alpha/2}f}_{\ell_2(\mathcal{W})},
\end{equation}
where $L$ is the Laplacian on $\mathcal{G}$.

Note that since $L$ is real and symmetric, there is a real orthogonal matrix $P$ such that 
\begin{equation}\label{eq:orthogonal_decomp}
L
=
P^TDP
\end{equation} 
and $D$ is the diagonal matrix of eigenvalues \cite[Chapter 4]{horn12}.  Hence, $L^{\alpha/2}=P^TD^{\alpha/2}P$.

The Sobolev space $H_{2}^{\alpha}(\mathcal{G})$ is defined by the norm
\begin{equation}
	\norm{f}_{H_{2}^{\alpha}(\mathcal{G})}
	=
	\norm{f}_{\ell_2(\mathcal{G})}
	+
	\abs{f}_{H_{2}^{\alpha}(\mathcal{G})}.
\end{equation}
Note that every function is included in these spaces on a finite domain. The Sobolev norm allows us to measure the smoothness of a function, and larger values of $\alpha$ penalize large jumps more than smaller values. This is analogous to the scale of Sobolev spaces on continuous domains.

\section{Conditionally positive definite basis functions}
\label{sec:basis_functions}

On Euclidean spaces, polyharmonic splines are defined as Green's functions of powers of the Laplacian.  
On the graph, since the Laplacian has a zero eigenvalue, we use the column vectors of powers of $L^{\dagger}$, the Moore-Penrose pseudo-inverse of the Laplacian, as our Green's functions.  Using \eqref{eq:orthogonal_decomp}, the pseudo-inverse is 
\begin{equation}
L^{\dagger}
=
P^TD^{\dagger}P,
\end{equation}
where $D^{\dagger}$ is the diagonal matrix containing the multiplicative inverses of the non-zero entries of $D$. We refer the reader to \cite[Chapter 7]{horn12} for additional details.

\begin{definition}
On a graph $\mathcal{G}$, the order $2\alpha$ polyharmonic spline centered at $v_k\in\mathcal{V}$ is denoted as $\Phi_{\alpha}(\cdot,v_k)$. It is the $k$th column of $(L^{\dagger})^{\alpha}$.
\end{definition}

An important property of polyharmonic functions on $\mathbb{R}^d$ is that they are conditionally positive definite \cite{wendland05}, which means that one can use them to construct interpolants to scattered data.  Here we state this result for our polyharmonic splines on graphs.

\begin{proposition}
\label{prop:interpolant}
Given a graph $\mathcal{G}$, a collection of distinct vertices $\widetilde{\mathcal{V}} \subset \mathcal{V}$ 
and data 
$F = \set{f_{v}}_{v\in \widetilde{\mathcal{V}}} \subset \mathbb{R}$, we can form the interpolant
\begin{equation}\label{eq:interpolant}
	s_{F,\widetilde{\mathcal{V}}}
	=
	C\Lambda_0
	+
	\sum_{v\in \widetilde{\mathcal{V}}} \beta_{v} \Phi_{\alpha}(\cdot, v).
\end{equation}
satisfying
\begin{equation}
	f_{v_0}
	=
	C\Lambda_0
	+
	\sum_{v\in \widetilde{\mathcal{V}}} \beta_{v} \Phi_{\alpha}(v_0, v),
\end{equation}
for every $v_0\in\widetilde{\mathcal{V}}$. The vector of coefficients $\beta$ is orthogonal to 
$\Lambda_0$ restricted to $\widetilde{\mathcal{V}}$. Note that $\beta$, $C$, and hence the interpolant $s_{F,\widetilde{\mathcal{V}}}$ will depend on $\alpha$. 
\end{proposition}
\begin{proof}
The Laplacian has a single zero eigenvalue with corresponding eigenvector $\Lambda_0$. The same is true for the pseudo-inverse of the Laplacian. Also note that the pseudo-inverse is conditionally positive definite. Hence the submatrix corresponding to the rows and columns associated with $\widetilde{\mathcal{V}}$ is also conditionally positive definite. 
\end{proof}

In the previous proposition, the coefficients $\beta_{w}$ and constant $C$ are found by solving a matrix equation involving an augmented interpolation matrix \cite[Section 8.5]{wendland05}.

\subsection{Variational characterization}

Another important property of conditionally positive definite functions is a variational characterization. In particular, an interpolant of the form \eqref{eq:interpolant} should have the smallest norm over the collection of all interpolants in a particular semi-Hilbert space. The semi-Hilbert space in question is the reproducing kernel Hilbert space (or \emph{native space}) associated with 
the kernel $(L^{\dagger})^{\alpha}$. 

\begin{definition}
The native space $\mathcal{N}_{\alpha}$ for the kernel $(L^{\dagger})^{\alpha}$ is defined by the semi-inner product 
\begin{equation}
	\innprod{f}{g}_{\mathcal{N}_{\alpha}}
	=
	\innprod{L^{\alpha/2}f}{L^{\alpha/2}g}_{\ell_2(\mathcal{G})},
\end{equation}
where $\innprod{\cdot}{\cdot}_{\ell_2(\mathcal{G})}$ is the standard inner product on $\ell_2(\mathcal{G})$. The semi-norm for $\mathcal{N}_{\alpha}$ is denoted as $\abs{\cdot}_{H_{2}^{\alpha}(\mathcal{G})}$.
\end{definition}

\begin{proposition}
Under the assumptions of \cref{prop:interpolant},  the interpolant $s_{F,\widetilde{\mathcal{V}}}$ has minimal native space semi-norm over all interpolants in $\mathcal{N}_{\alpha}$.
\end{proposition}
\begin{proof}
The proof of this result is analogous to that of \cite[Theorem 13.2]{wendland05}. 
We provide the details for the benefit of the reader.

First, suppose $g:\mathcal{V} \rightarrow \mathbb{R}$ is $0$ on  $\widetilde{\mathcal{V}}$.  
We have
\begin{align}
\begin{split}
	\innprod{g}{s_{F,\widetilde{\mathcal{V}} }}_{\mathcal{N}_{\alpha}}
	&=
	\innprod{g}
	{ \sum_{v\in \widetilde{\mathcal{V}}} \beta_{v} \Phi_{\alpha}(\cdot, v)  }
	_{\mathcal{N}_{\alpha}}	\\
	&=
	\sum_{v\in \widetilde{\mathcal{V}}} \beta_{v} g(v)\\
	&=
	0.
\end{split}
\end{align}
Hence for any interpolant $s \in \mathcal{N}_{\alpha}$,
\begin{align}
\begin{split}
	\abs{s_{F,\widetilde{\mathcal{V}}}}_{\mathcal{N}_{\alpha}}^2
	&=
	\innprod{s_{F,\widetilde{\mathcal{V}}}}{s_{F,\widetilde{\mathcal{V}}} }_
	{\mathcal{N}_{\alpha}} \\
	&= 
	\innprod{s_{F,\widetilde{\mathcal{V}}}}{s_{F,\widetilde{\mathcal{V}}}-s+s }_
	{\mathcal{N}_{\alpha}} \\
	&=
	\innprod{s_{F,\widetilde{\mathcal{V}}}}{s }_
	{\mathcal{N}_{\alpha}} \\
	&\leq
	\abs{s_{F,\widetilde{\mathcal{V}}}}_{\mathcal{N}_{\alpha}}
	\abs{s}_{\mathcal{N}_{\alpha}}.
\end{split}
\end{align}
Dividing by $\abs{s_{F,\widetilde{\mathcal{V}}}}_{\mathcal{N}_{\alpha}}$ results in
\begin{equation}
\abs{s_{F,\widetilde{\mathcal{V}}}}_{\mathcal{N}_{\alpha}}
\leq
\abs{s}_{\mathcal{N}_{\alpha}}.
\end{equation}
\end{proof}

\section{Polyharmonic Lagrange functions}
\label{sec:lagrange}

We are now in a position to define the polyharmonic Lagrange functions, which are the primary objects of study in this paper. Similar to the classical polynomial Lagrange functions, they are bases for constructing interpolants at a collection of 
vertices $\widetilde{\mathcal{V}} \subset \mathcal{V}$. The Lagrange function $\chi(\cdot, v_0)$ centered at $v_{0}\in\widetilde{\mathcal{V}}$ is a linear combination  of polyharmonic splines
\begin{equation}
	\chi(\cdot, v_{0})
	=
	C\Lambda_0
	+
	\sum_{v\in \widetilde{\mathcal{V}}} \beta_{v} \Phi_{\alpha}(\cdot,v).
\end{equation}
that satisfies
\begin{equation}
	\chi(v, v_{0})
	=
	\begin{cases}
	1, & v = v_0\\
	0, & v\in \widetilde{\mathcal{V}} \backslash \set{v_0}
	\end{cases}.
\end{equation}
The Lagrange functions depend on the parameter $\alpha$. 

\subsection{Decay estimates for Lagrange functions}

Our interest in Lagrange functions is based on their usefulness in constructing approximations to functions $f:\mathcal{V}\rightarrow \mathbb{R}$ that are only known on a subset $\widetilde{\mathcal{V}} \subset\mathcal{V}$. In particular,
\begin{enumerate}[i)]
\item
The form of the approximation is very simple:
\begin{equation}
	\sum_{v \in \widetilde{\mathcal{V}}} f(v) \chi(\cdot,v)
\end{equation}
\item 
As we shall see, the polyharmonic Lagrange functions are well localized, as are their continuous counterparts \cite{hangelbroek12prk}. 
\end{enumerate}

In applications, localized bases are important for several reasons. For example,
errors in the acquired sample data are confined to a small region. Also, if new data is acquired, the approximation can be updated locally.

In order to establish the decay of the Lagrange functions, we use a bulk chasing argument similar to the one used in \cite{matveev92}. This argument makes use of a zeros lemma, which we state as an assumption to preserve generality. Afterward, we verify it for certain classes of graphs.  Essentially, it requires the interpolation vertices $\widetilde{\mathcal{V}}$ to be dense in $\mathcal{G}$, where we measure density in terms of the fill distance.

\begin{definition}
The fill distance $h$ of a subset $\widetilde{\mathcal{V}}\subset \mathcal{V}$ is defined as
\begin{equation}
	h 
	:=
	\max_{v\in \mathcal{V}}
	\min_{\tilde{v} \in \widetilde{\mathcal{V}} }
	\rho(v,\tilde{v}).
\end{equation}
\end{definition}

We now state the assumption, which says that for functions with lots of zeros, the $\ell_2$ norm is controlled by the Sobolev semi-norm. A continuous domain version can be found in \cite[Section 7.4]{natterer86}.

\begin{assumption}\label{as:zero_lem}
For a graph 
$\mathcal{G} 
= 
\set{
\mathcal{V},
\mathbf{E}\subset \mathcal{V} \times \mathcal{V}, 
w,\rho
}$, 
we assume that there is a nonempty subset 
$\widetilde{\mathcal{V}}\subset \mathcal{V}$ with fill distance $h$ satisfying the following property.
There exists a constant $C>0$ (depending on $\widetilde{\mathcal{V}}$) such that for every $f:\mathcal{V}\rightarrow \mathbb{R}$ that is zero on  
$\widetilde{\mathcal{V}}$ and for every $\mathcal{V}_1 \subset \mathcal{V}$ 
\begin{equation}\label{eq:assumption_zeros}
	\norm{f}_{\ell_2(\mathcal{V}_1)   }^2
	\leq
	C
	\abs{f}_{H_2^{2}( \bar{\mathcal{V}}_1  ) }^2,
\end{equation}
where 
\begin{equation}
	\bar{\mathcal{V}}_1
	:=
	\setb{v\in \mathcal{V}}
	{\rho\parenth{v,\mathcal{V}_1}\leq 2h}.
\end{equation}
\end{assumption}
A non-local version of this result for general graphs is provided in \ref{sec:zero_lem}. 
Let us also point out that \cref{eq:assumption_zeros} is similar to a Poincare inequality \cite{pesenson08,pesenson10s,fuhr13}. There the authors give several estimates for such inequalities on various type of graphs.

To verify the assumption above, we make use of coverings of a graph by subgraphs satisfying certain properties, primarily Dirichlet boundary condition. Properties of the Laplacian on such subgraphs can be found in \cite{chung97,grigor18}. In particular, we are interested in the eigenvalues of the submatrix of the Laplacian corresponding to these subgraphs. Bounds on the smallest eigenvalue, called the Dirichlet eigenvalue for Dirichlet boundary conditions, can be found in terms of the Cheeger constant of the subgraph  \cite{chung97,grigor18}.

\begin{proposition}\label{prop:as_zeros_circ}
\cref{as:zero_lem} is valid for functions on cycle graphs, independent of the distance function $\rho$.
\end{proposition}
\begin{proof}
Let $\mathcal{G}$ be a cycle graph with vertices
$\mathcal{V}=\set{v_{n}}_{k=0}^{N-1}$
with $N\geq 3$. Assume the vertices are labeled such that $v_0$ is adjacent to $v_1$, $v_1$ is adjacent to $v_2$, \dots, $v_{N-2}$ is adjacent to $v_{N-1}$ and  $v_{N-1}$ is adjacent to $v_{0}$. Denote the vertices of 
$\widetilde{\mathcal{V}}$ as $\set{v_{n_k}}_{k=0}^{N_1-1}$ where $n_k<n_{k+1}$ for all $k$, and let $f$ be zero on $\widetilde{\mathcal{V}}$.

The idea for the proof is to cover $\mathcal{V}_1$ (in fact we cover $\mathcal{G}$ to remove the dependence on  $\mathcal{V}_1$) by subgraphs $\mathcal{G}_k$, and show that the inequality holds on these smaller regions.  We define the subgraphs so that they do not overlap too much, and therefore the inequality holds on unions of the $\mathcal{G}_k$.

For simplicity, suppose $N_1$ is even; our argument can be suitably modified otherwise.  Define the following induced subgraphs (which are paths connecting every other zero of $f$):
\begin{align}
\begin{split}
	\mathcal{G}_0 
	&= 
	\set{v_{n_0},v_{n_0+1},\dots, v_{n_2} }\\
	\mathcal{G}_1 
	&= 
	\set{v_{n_1},v_{n_1+1},\dots, v_{n_3} }\\
	&\vdots\\
	\mathcal{G}_{N_1-3} 
	&= 
	\set{v_{n_{N_1-3}},v_{n_{N_1-3}+1},\dots, v_{n_{N_1-1}} }\\
	\mathcal{G}_{N_1-2}
	&= 
	\set{v_{n_{N_1-2}},v_{n_{N_1-2}+1},\dots, v_{N-1},v_0,\dots, v_{n_0} }\\
	\mathcal{G}_{N_1-1}
	&= 
	\set{v_{n_{N_1-1}},v_{n_{N_1-1}+1},\dots, v_{N-1},v_0,\dots, v_{n_1} }.
\end{split}
\end{align}
The length of each path 
$\mathcal{G}_k$ is at most $4h$, where $h$ is the fill distance. For any given set 
$\mathcal{V}_1$, there is a union of $\mathcal{G}_k$'s such that the union of their interiors (denoted 
$\mathrm{int}\parenth{\mathcal{G}_k}$) cover $\mathcal{V}_1$ and is contained in $\bar{\mathcal{V}}_1$. Let $\mathcal{K}$ be the index set for such a collection; i.e.,
\begin{equation}
	\mathcal{V}_1
	\subseteq
	\bigcup_{k\in \mathcal{K}}
	\mathrm{int}\parenth{\mathcal{G}_k}
	\subseteq
	\bar{\mathcal{V}}_1
\end{equation}
Also, note that each vertex is in at most two paths 
$\mathrm{int}\parenth{\mathcal{G}_k}$.

The subgraphs $\mathcal{G}_k$ were constructed so that  the function $f$ satisfies Dirichlet boundary conditions on each one.  
For each $k$, consider the submatrix 
$L_{\mathrm{int} \parenth{\mathcal{G}_k}}$ of the Laplacian, where the rows and columns correspond to the vertices of 
$\mathrm{int} \parenth{\mathcal{G}_k}$.  
The minimal eigenvalue of this matrix (the Dirichlet eigenvalue \cite{chung97}) is positive. 
We denote the Dirichlet eigenvalue on $\mathrm{int}\parenth{\mathcal{G}_k}$ as 
$\lambda_0^{\mathcal{G}_k}$. 
 
We now have
\begin{align}
\begin{split}
	\norm{f}_{\ell_2(\mathcal{V}_1)}^2
	&\leq
	\sum_{k\in \mathcal{K}}
	\norm{f}_{\ell_2( \mathrm{int} \parenth{\mathcal{G}_k} )}^2\\
	&\leq
	\sum_{k\in \mathcal{K}}
	\parenth{
	\frac{1}{\lambda_0^{\mathcal{G}_k}}}^2
	\norm{L f}_
	{\ell_2( \mathrm{int} \parenth{\mathcal{G}_k} )}^2\\
	&\leq
	\max_{0 \leq k\leq N_1-1} 
	\set{ \parenth{\frac{1}{\lambda_0^{\mathcal{G}_k}}}^2}
	\sum_{k\in \mathcal{K}}
	\norm{L f}_
	{\ell_2( \mathrm{int} \parenth{\mathcal{G}_k} )}^2\\
	&\leq
	2
	\max_{0 \leq k\leq N_1-1} 
	\set{ \parenth{\frac{1}{\lambda_0^{\mathcal{G}_k}}}^2}
	\norm{L f}_
	{\ell_2( \bar{\mathcal{V}}_1 )}^2.
\end{split}
\end{align}
\end{proof}

In the proof of the previous proposition, we derived a constant $C$ for \eqref{eq:assumption_zeros}. The constant depends only on the Dirichlet eigenvalues for the subgraphs $\mathcal{G}_k$, which were constructed to fit the density of the vertices $\widetilde{\mathcal{V}}$ in $\mathcal{V}$. In particular, this means that
the constant does not depend on the size of the graph.  
For example, if a cycle graph has equally weighted edges and bounded fill distance $h$, then the constant $C$ will have a bound that is independent of the number of vertices $N$. In this case, the Dirichlet eigenvalues are actually known, cf. Lemma I.14 of  
\cite{ellis02}.

In order to verify \cref{as:zero_lem}, we believe that it is necessary to have a covering of the graph $\mathcal{G}$ by subgraphs satisfying some boundary conditions. In \cref{prop:as_zeros_circ}, the subgraphs satisfy Dirichlet boundary condition. This is relevant to certain machine learning applications where training data sets are larger than the collection of unknowns being evaluated.
Also note that one could consider other boundary conditions such as Neumann boundary conditions.
The next proposition describes how \cref{as:zero_lem} is satisfied by certain graphs constructed for machine learning.

\begin{proposition}\label{prop:as_zeros_gen}
For a graph $\mathcal{G}$, suppose the vertices are divided into two sets 
$\mathcal{V}_k$ and $\mathcal{V}_u$ corresponding to known and unknown values for some function on $\mathcal{G}$, and suppose that there are no edges connecting distinct vertices in $\mathcal{V}_u$. If the length of the edges in $\mathcal{G}$ are bounded below by 
$\rho_{\text{max}}/2$ and $w=\rho^{-1}$ on edges within the graph, then $\mathcal{G}$ satisfies \cref{as:zero_lem} with 
$\widetilde{\mathcal{V}}=\mathcal{V}_k$.
\end{proposition}
\begin{proof}
The idea for the proof is the same as for cycle graphs. We need to define a cover of $\mathcal{G}$ by subgraphs that have bounded Dirichlet eigenvalues. Moreover, the subgraphs cannot be too large; in addition to covering 
$\mathcal{V}_1$, they must lie within 
$\bar{\mathcal{V}}_1$. This is the reason for the restriction on the edge distances. The fact that the minimum eigenvalue of the Laplacian on these subgraphs is bounded will follow from a bound on the Cheeger constant for the subgraph \cite[Chapters 3,4]{grigor18}.

We construct a subgraph, starting from a given vertex $v_0$ and grow the subgraph until it is surrounded by vertices from 
$\mathcal{V}_k$.  We introduce the notation 
$\mathcal{N}_v$ for the neighbors of a vertex $v$.
To make our argument precise, let $v_0\in \mathcal{V}$, and consider the subgraph $\mathcal{G}_{v_0}$ consisting of the vertices
\begin{equation}
\Omega_{v_0}
:=
\set{v_0}
\bigcup
\mathcal{N}_{v_0}
\bigcup
\parenth{
\bigcup_{v\in \mathcal{N}_{v_0}\bigcap \mathcal{V}_u }
\mathcal{N}_{v}
},
\end{equation}
i.e. $\Omega_{v_0}$ consists of the neighbors of $v_0$ and the neighbors of the neighbors that are in $\mathcal{V}_u$. As no two vertices of  
$\mathcal{V}_u$ are connected by an edge, the boundary of 
int$(\mathcal{G}_{v_0})$ consists entirely of vertices from 
$\mathcal{V}_k$. 

To bound the Cheeger constant, and hence the minimum eigenvalue, on $\text{int}(\mathcal{G}_{v_0})$, we note the following:
\begin{enumerate}[a)]
\item
The number of vertices in $\mathcal{G}_{v_0}$, depending on the maximum degree $M$, is at most $1+M^2$
\item 
Given a nonempty subset $U$ of  
$\text{int}(\mathcal{G}_{v_0})$, the sum of the edge weights between $U$ and its boundary in $\mathcal{G}$ is at least 
$\rho_{\text{max}}^{-1}$. Also, the sum of the edge weights (counting multiplicities) over the vertices of $U$ is at most 
$(M+1)M2\rho_{\text{max}}^{-1}$
\end{enumerate}

Using these facts, we find a lower bound for the Cheeger constant for subgraphs of this type \cite[Section 4.2]{grigor18}. In particular the Cheeger constant $\bar{h}$ satisfies
\begin{align}
\begin{split}
\bar{h}(\text{int}(\mathcal{G}_{v_0}))
&\geq
\frac{\rho_{\text{max}}^{-1}}{(M+1)M2\rho_{\text{max}}^{-1}}\\
&=
\frac{1}{(M+1)M2}
\end{split}
\end{align}
Cheeger's inequality then implies that the minimum eigenvalue 
$\lambda_0^{\mathcal{G}_{v_0}}$ of
$L_{\mathrm{int} \parenth{\mathcal{G}_{v_0}}}$ satisfies
\begin{align}
\begin{split}
\lambda_0^{\mathcal{G}_{v_0}}
&\geq
\frac{1}{2}
\bar{h}(\text{int}(\mathcal{G}_{v_0}))^2\\
&\geq
\frac{1}{8}\parenth{\frac{1}{(M+1)M}}^2,
\end{split}
\end{align}
so
\begin{align}
\parenth{\frac{1}{\lambda_0^{\mathcal{G}_{v_0}}}}^2
&\leq
64(M+1)^4M^4.
\end{align}

Now, given a set $\mathcal{V}_1$, we cover it by 
\begin{equation}
\bigcup_{v\in\mathcal{V}_1} \text{int}(\mathcal{G}_{v}) 
\subseteq
\bar{\mathcal{V}}_1. 
\end{equation}
Containment in $\bar{\mathcal{V}}_1$ follows from the fact that 
$\rho_{\text{max}}\leq 2(\rho_{\text{max}}/2)\leq 2h$. We also note that the construction of $\mathcal{G}_{v_0}$ means that any given vertex $v\in \mathcal{G}$ can be in the interior of at most $M$ subgraphs of this type, and the number of vertices in 
$\mathcal{G}_{v_0}$, depending on the maximum degree $M$, is at most $1+M^2$.

We now 
finish the proof. Given any function $f$ that is zero on 
$\mathcal{V}_k$, we have
\begin{align}
\begin{split}
	\norm{f}_{\ell_2(\mathcal{V}_1)}^2
	&\leq
	\sum_{v\in \mathcal{V}_1}
	\norm{f}_{\ell_2( \mathrm{int} \parenth{\mathcal{G}_{v}} )}^2\\
	&\leq
	\sum_{v\in \mathcal{V}_1}
	\parenth{
	\frac{1}{\lambda_0^{\mathcal{G}_v}}}^2
	\norm{L f}_
	{\ell_2( \mathrm{int} \parenth{\mathcal{G}_v} )}^2\\
	&\leq
	64(M+1)^4M^4
	\sum_{v\in \mathcal{V}_2}
	\norm{L f}_
	{\ell_2( \mathrm{int} \parenth{\mathcal{G}_v} )}^2\\
	&\leq
	64(M+1)^4M^5
	\norm{L f}_
	{\ell_2( \bar{\mathcal{V}}_1 )}^2.
\end{split}
\end{align}
\end{proof}

We now proceed to a key lemma for deriving the decay of the Lagrange functions.

\begin{lemma}\label{lem:bulk_1}
Let 
$\mathcal{G} 
= 
\set{
\mathcal{V},
\mathbf{E}\subset \mathcal{V} \times \mathcal{V}, 
w,\rho
}$  
and  $\widetilde{\mathcal{V}} \subset \mathcal{V}$ satisfy the hypotheses of 
\cref{as:zero_lem}. 
Let $v_{0}\in \widetilde{\mathcal{V}}$, and consider the  Lagrange function $\chi(\cdot,v_{0})$ associated with 
$(L^{\dagger})^2$ and the vertex set $\widetilde{\mathcal{V}}$ with fill distance $h$. 
If $0 < 3\rho_{\text{max}}+2h < r_2 < r_3 <\infty$, then there is a constant $\mu<1$ such that 
\begin{equation}
	\abs{\chi}_{H_2^{2}(\mathcal{G}\backslash B(v_{n_0};r_4 )  ) }
	\leq
	\mu
	\abs{\chi}_{H_2^{2}(\mathcal{G}\backslash B(v_{n_0};r_1 )  ) },
\end{equation}
where $r_1:= r_2-2\rho_{\text{max}}-2h$ and $r_4:=r_3+2h$.
\end{lemma}

\begin{proof}
Our proof is similar to that of 
\cite{hangelbroek10,hangelbroek12plc}. In this proof, all balls and annuli are centered at $v_{0}$. We simplify our notation by omitting the center: for example we write $B(r_0)$ rather than $B(v_{0};r_0)$.

Let $\phi$ be a function satisfying
\begin{itemize}
\item $0\leq \phi(v)\leq 1$
\item $\phi(v)=1$ for $v\in B(r_2)$
\item $\phi(v)=0$ for $v\notin B(r_3)$
\end{itemize}

Recall the interpolation criterion for $\chi$
\begin{equation}
	\chi(v, v_{0})
	=
	\begin{cases}
	1, & v = v_{0}\\
	0, & v\in \widetilde{\mathcal{V}}\backslash \set{v_{0}}
	\end{cases},
\end{equation}
and notice that the product $\phi\chi$ satisfies the same interpolation conditions. Hence the variational property of $\chi$ implies
\begin{align}
\begin{split}
	\abs{\chi}_{H_2^{2}(\mathcal{G})}^2
	&\leq
	\abs{\phi \chi}_{H_2^{2}(\mathcal{G})}^2 \\
	&=
	\abs{ \chi}_{H_2^{2}(B(r_2-\rho_{\text{max}}))}^2
	+
	\abs{\phi \chi}_{H_2^{2}(A(r_2-\rho_{\text{max}},r_3))}^2.
\end{split}
\end{align}
Using the estimate above, we have 
\begin{align}
\begin{split}
	\abs{\chi}_{H_2^{2}(\mathcal{G} \backslash B(r_2))}^2
	&=
	\abs{ \chi}_{H_2^{2}(\mathcal{G})}^2
	-
	\abs{ \chi}_{H_2^{2}(B(r_2))}^2\\
	&\leq
	\parenth{
	\abs{ \chi}_{H_2^{2}(B(r_2-\rho_{\text{max}}))}^2
	+
	\abs{\phi \chi}_{H_2^{2}(A(r_2-\rho_{\text{max}},r_3))}^2
	}
	-
	\abs{ \chi}_{H_2^{2}(B(r_2))}^2.
\end{split}
\end{align}
Simplifying the expression gives

\begin{align}
\begin{split}
	\abs{\chi}_{H_2^{2}(\mathcal{G} \backslash B(r_2))}^2
	&\leq
	\parenth{
	\abs{ \chi}_{H_2^{2}(B(r_2-\rho_{\text{max}}))}^2
	-
	\abs{ \chi}_{H_2^{2}(B(r_2))}^2
	}
	+
	\abs{\phi \chi}_{H_2^{2}(A(r_2-\rho_{\text{max}},r_3))}^2\\
	&\leq 
	\abs{\phi \chi}_{H_2^{2}(A(r_2-\rho_{\text{max}},r_3))}^2
\end{split}
\end{align}
Using the definition of the semi-norm and properties of $L$,
\begin{align}
\begin{split}
	\abs{\chi}_{H_2^{2}(\mathcal{G} \backslash B(r_2))}^2
	&\leq
	\norm{L \parenth{\phi\chi} }_{\ell_2(A(r_2-\rho_{\text{max}},r_3))}^2 \\
	&\leq
	4
	\norm{\phi\chi}_{\ell_2(A(r_2-2\rho_{\text{max}},r_3))} ^2 \\
	&\leq 
	4 \norm{\chi}_{\ell_2(A(r_2-2\rho_{\text{max}},r_3))} ^2.
\end{split}
\end{align}
Now we apply \cref{as:zero_lem} with 
$f=\chi$, 
$\mathcal{V}_1 = A(r_2-2\rho_{\text{max}},r_3) $,
and 
$\bar{\mathcal{V}}_1 = A(r_2-2\rho_{\text{max}}-2h,r_3+2h)=A(r_1,r_4)$. 
This implies that there is a constant $C_0$ such that
\begin{align}
\begin{split}
	\abs{\chi}_{H_2^{2}(\mathcal{G} \backslash B(r_4))}^2
	&\leq
	\abs{\chi}_{H_2^{2}(\mathcal{G} \backslash B(r_2))}^2\\
	&\leq
	4 C_0 \abs{\chi}_{H_2^{2}(A(r_1,r_4))}^2.
\end{split}
\end{align}

Writing the annulus as a set difference
\begin{align}
	\abs{\chi}_{H_2^{2}(\mathcal{G} \backslash B(r_4))}^2
	&\leq
	4C_0
	\parenth{
	\abs{ \chi}_{H_2^{2}(\mathcal{G}\backslash B(r_1))}^2
	-
	\abs{ \chi}_{H_2^{2}(\mathcal{G}\backslash B(r_4))}^2
	},
\end{align}
which implies
\begin{equation}
	(1+4C_0)\abs{ \chi}_{H_2^{2}(\mathcal{G}\backslash B(r_4))}^2
	\leq
	4C_0 \abs{ \chi}_{H_2^{2}(\mathcal{G}\backslash B(r_1))}^2.
\end{equation}
This is equivalent to
\begin{equation}
	\abs{ \chi}_{H_2^{2}(\mathcal{G}\backslash B(r_4))}^2
	\leq
	\frac{4C_0}{1+4C_0} 
	\abs{ \chi}_{H_2^{2}(\mathcal{G}\backslash B(r_1))}^2,
\end{equation}
so $\mu=4C_0/(1+4C_0)$.
\end{proof}

The significance of $\mu$ in the previous lemma and following theorem is that it specifies the exponential decay rate of Lagrange functions: smaller values of $\mu$ imply faster decay. A similar constant appears in the case of a continuous domain \cite{hangelbroek12prk,matveev92}.

\begin{theorem}\label{thm:decay_lagrange}
Under the assumptions of \cref{lem:bulk_1}, the Lagrange function $\chi(\cdot,v_{0})$ satisfies exponential decay away from its center. In particular,  
there are constants $C,T>0$ such that 
\begin{equation}
	\abs{\chi(v_{1},v_{0})} 
	\leq
	C
	\mu^{T \rho(v_{1},v_{0})}\abs{\chi(\cdot,v_{0})}_{H_2^2(\mathcal{G})},
\end{equation}
where $\mu<1$ comes from \cref{lem:bulk_1}.
\end{theorem}
\begin{proof}
Let $v_{0}\in \widetilde{\mathcal{V}}$ and $v_{1}\in \mathcal{V}$ such that
$\rho(v_{0}, v_{1})$ satisfies 
\begin{equation}
M(4h+3\rho_{\text{max}})+2h+2\rho_{\text{max}}
\leq
\rho(v_{0}, v_{1})
\leq
M(4h+3\rho_{\text{max}})+2h+3\rho_{\text{max}}
\end{equation}
for some positive integer $M$.
We now define a sequence of balls, centered at $v_{0}$, whose radii are shrinking by an amount prescribed by \cref{lem:bulk_1}
\begin{equation}
B_k 
:=
B(v_{0},(M-k)(4h+3\rho_{\text{max}})+2h+2\rho_{\text{max}})
\end{equation}
for $k=0,\dots,M.$

Then
\begin{align}
\begin{split}
\abs{\chi(v_{1},v_{0})} 
&\leq
C_0^{1/2}\abs{\chi(\cdot,v_{0})}_{H_2^2(B(v_{1},2h))}\\
&\leq
C_0^{1/2}\abs{\chi(\cdot,v_{0})}_{H_2^2(\mathcal{G}\backslash B_0)}
\end{split}
\end{align}
where $C_0$ comes from \cref{as:zero_lem}.

Applying \cref{lem:bulk_1}
\begin{align}\label{eq:decay_precise}
\begin{split}
\abs{\chi(v_{1},v_{0})} 
&\leq
C_0^{1/2}\mu \abs{\chi(\cdot,v_{0})}_{H_2^2(\mathcal{G}\backslash B_1)}\\
&\leq
C_0^{1/2}\mu^M \abs{\chi(\cdot,v_{0})}_{H_2^2(\mathcal{G}\backslash B_M)}\\
&\leq
C_0^{1/2}\mu^{\frac{\rho(v_{1},v_{0})-2h-3\rho_{\text{max}}}{4h+3\rho_{\text{max}}}} 
\abs{\chi(\cdot,v_{0})}_{H_2^2(\mathcal{G})} \\
&=
\parenth{C_0^{1/2}\mu^{-\frac{2h+3\rho_{\text{max}}}{4h+3\rho_{\text{max}}}}}
\mu^{(4h+3\rho_{\text{max}})^{-1} \rho(v_{1},v_{0}) }
\abs{\chi(\cdot,v_{0})}_{H_2^2(\mathcal{G})},
\end{split}
\end{align}
so in particular,  
$C=C_0^{1/2}\mu^{-\frac{2h+3\rho_{\text{max}}}{4h+3\rho_{\text{max}}}}$
and
$T=(4h+3\rho_{\text{max}})^{-1}$.
\end{proof}

Let us point out that, in addition to exponential decay, we have shown more precisely how the Lagrange functions decay in terms of the fill distance $h$ and the maximum distance between points $\rho_{\text{max}}$. This can be seen in \eqref{eq:decay_precise}. 
As $h$ and $\rho_{\text{max}}$ decrease, the basis functions become more localized. 
This is analogous to B-splines on euclidean spaces that scale with the density of the sampling grid.

On infinite graphs, the Moore-Penrose pseudo-inverse of the Laplacian is known to satisfy exponential decay \cite[Theorem 5.1]{sun07}, and exponential decay of our Lagrange functions could be derived accordingly in that setting. However, we would not be able to obtain the more precise decay of \cref{thm:decay_lagrange}. In fact, it is clear from the examples in the next section that the Lagrange functions on graphs are much more localized than the polyharmonic functions, just as in the case of polyharmonic functions on continuous domains. 

As a corollary of our theorem, we can deduce the exponential decay of the coefficients of the Lagrange functions.  This follows from the next proposition that relates the coefficients to a native space semi-inner product between two Lagrange functions.  

\begin{proposition}\label{prop:coeff_innprod}
Let $v_{0},v_{1}\in\widetilde{\mathcal{V}}\subset\mathcal{V}$ and consider the Lagrange functions associated with $L^{\alpha}$ for $\alpha>0$
\begin{align}
\begin{split}
\chi(\cdot,v_{0})
&=
C_0 \Lambda_0
+
\sum_{v\in\widetilde{\mathcal{V}}} 
\gamma_v \Phi_{\alpha}(\cdot,v)\\
\chi(\cdot,v_{1})
&=
C_1 \Lambda_0
+
\sum_{v\in\widetilde{\mathcal{V}}} 
\beta_v \Phi_{\alpha}(\cdot,v)
\end{split}
\end{align}
Then we have
\begin{equation}
\innprod{\chi(\cdot,v_{0})}{\chi(\cdot,v_{1})}_{\mathcal{N}_{\alpha}} 
=
\gamma_{v_{1}}
=
\beta_{v_{0}}
\end{equation}
\end{proposition}
\begin{proof}
This follows from the definition of the native space norm and the symmetry of the Laplacian.
\begin{align}
\begin{split}
\innprod{\chi(\cdot,v_{0})}{\chi(\cdot,v_{1})}_{\mathcal{N}_{\alpha}} 
&=
\innprod{L^{\alpha}\chi(\cdot,v_{0})}{\chi(\cdot,v_{1})}_{\ell_2(\mathcal{G})} \\
&=
\innprod{L^{\alpha}\parenth{C_0\Lambda_0+\sum_{v\in\widetilde{\mathcal{V}}} 
\gamma_v \Phi_{\alpha}(\cdot,v)}}{\chi(\cdot,v_{1})}_{\ell_2(\mathcal{G})} \\
&=
\innprod{\sum_{v\in\widetilde{\mathcal{V}}}
\gamma_v L^{\alpha} \Phi_{\alpha}(\cdot,v)}{\chi(\cdot,v_{1})}_{\ell_2(\mathcal{G})} \\
&=
\innprod{\sum_{v\in\widetilde{\mathcal{V}}}
\gamma_v e_v}{\chi(\cdot,v_{1})}_{\ell_2(\mathcal{G})} \\
&=
\sum_{v\in\widetilde{\mathcal{V}}}
\gamma_v
\innprod{ e_v}{\chi(\cdot,v_{1})}_{\ell_2(\mathcal{G})} \\
&=
\gamma_{v_{1}}
\end{split}
\end{align}
By symmetry, this is also equal to $\beta_{v_{0}}$.
\end{proof}

\begin{corollary}
The coefficients of the Lagrange functions associated with 
$(L^{\dagger})^2$ decay exponentially fast.
\end{corollary}
\begin{proof}
Let $v_{0},v_{1}\in\widetilde{\mathcal{V}}$. Considering  \cref{prop:coeff_innprod},  to show exponential decay of the coefficients, it suffices to show the exponential decay of the semi-inner product 
$\innprod{\chi(\cdot,v_{0})}{\chi(\cdot,v_{1})}_{\mathcal{N}_{2}}$. Now, suppose
\begin{equation}
2M(4h+3\rho_{\text{max}})+2h+2\rho_{\text{max}}
\leq
\rho(v_{0}, v_{1})
\leq
2M(4h+3\rho_{\text{max}})+2h+3\rho_{\text{max}}
\end{equation}

Consider decomposing the graph into the following sets based on the distance\\ $r = M(4h+3\rho_{\text{max}})+2h+2\rho_{\text{max}}$.
We define the sets
\begin{align}
\begin{split}
X_0
&=
\setb{v\in\mathcal{V}}{\rho(v,v_{0})\leq r }\\
X_1
&=
\setb{v\in\mathcal{V}}{\rho(v,v_{1})> r }
\end{split}
\end{align}
Then
\begin{align}
\begin{split}
\innprod{\chi(\cdot,v_{0})}{\chi(\cdot,v_{1})}_{\mathcal{N}_{2}(\mathcal{G})}
&=
\innprod{\chi(\cdot,v_{0})}{\chi(\cdot,v_{1})}_{\mathcal{N}_{2}(X_0)}
+
\innprod{\chi(\cdot,v_{0})}{\chi(\cdot,v_{1})}_{\mathcal{N}_{2}(X_1)}\\
&\leq
\abs{\chi(\cdot,v_{0})}_{H_2^2(X_0)}\abs{\chi(\cdot,v_{1})}_{H_2^2(X_0)}
+
\abs{\chi(\cdot,v_{0})}_{H_2^2(X_1)}\abs{\chi(\cdot,v_{1})}_{H_2^2(X_1)}.
\end{split}
\end{align}
Let $\mu_0,\mu_1<1$ be the constants corresponding to the decay of
$\chi(\cdot,v_{0})$ and $\chi(\cdot,v_{1})$ respectively. Also, define $\mu=\max\{\mu_0,\mu_1\}$. Then 
\begin{align}
\begin{split}
\innprod{\chi(\cdot,v_{0})}{\chi(\cdot,v_{1})}_{\mathcal{N}_{2}(\mathcal{G})}
&\leq
\abs{\chi(\cdot,v_{0})}_{H_2^2(X_0)}\mu_1^M
+
\mu_0^M\abs{\chi(\cdot,v_{1})}_{H_2^2(X_1)}\\
&\leq
C \mu^{\frac{\rho(v_{1},v_{0})-2h-3\rho_{\text{max}}}{4h+3\rho_{\text{max}}}} 
\end{split}
\end{align}
\end{proof}

\section{Local Lagrange functions}
\label{sec:loc_lag}

Given the exponential decay of the Lagrange coefficients, we can truncate the expansion to form truncated Lagrange functions that closely approximate the full Lagrange functions. While these functions have a simpler form than the Lagrange functions, they require the same amount of computation to construct. The purpose of the truncated Lagrange functions is to provide a theoretical link between the Lagrange and local Lagrange functions defined below. Ultimately, we intend to use the more computationally efficient local Lagrange functions. 

The Lagrange function $\chi(\cdot, v_{0})$ centered at $v_{0}\in\widetilde{\mathcal{V}}$ is a linear combination  of polyharmonic splines
\begin{equation}
	\chi(\cdot, v_{0})
	=
	C\Lambda_0
	+
	\sum_{v\in \mathcal{V}} \beta_{v} \Phi_{\alpha}(\cdot, v).
\end{equation}
We form the truncated Lagrange function 
$\widetilde{\chi}(\cdot, v_{0})$
 in terms of a distance $K>0$.

 In particular, we define a neighborhood $B(v_{0},K)$ and consider
\begin{equation}
	\widetilde{\chi}(\cdot, v_{0})
	=
	C\Lambda_0
	+
	\sum_{v\in B(v_{n_0},K)\bigcap \widetilde{\mathcal{V}}} 
	\widetilde{\beta}_v
	\Phi_{\alpha}(\cdot, v).
\end{equation}
The coefficients are modified to ensure that they are orthogonal to $\Lambda_0$ restricted to the neighborhood.

The local Lagrange functions $\bar{\chi}(\cdot,v_{0})$ are constructed analogously to the Lagrange functions; however, they only use basis functions centered at points of $B(v_{0},K)$ and the interpolation conditions are only enforced on this ball.

\begin{definition} 
Given a neighborhood $B(v_{0},K)$ of $v_{0}\in\widetilde{\mathcal{V}}\subset\mathcal{V}$ with $K>0$,  we define the local Lagrange function
\begin{equation}
	\bar{\chi}(\cdot, v_{0})
	=
	\bar{C}\Lambda_0
	+
	\sum_{v\in B(v_{0},K) \bigcap \widetilde{\mathcal{V}}} 
	\bar{\beta}_{v} 
	\Phi_{\alpha}(\cdot, v),
\end{equation}
where 
\begin{equation}
\bar{\chi}(v, v_{0}) 
=
\begin{cases}
1, v = v_{0}\\
0, v \in B(v_{0},K)\bigcap \widetilde{\mathcal{V}} \backslash  \set{v_{0}}
\end{cases}
\end{equation}
and the vector of coefficients $\bar{\beta}_{v}$ is orthogonal to $\Lambda_0$ restricted to the neighborhood.
\end{definition}

\begin{conjecture}
The local Lagrange functions can be made arbitrarily close to the full Lagrange functions with a relatively small $K$ value. 
\end{conjecture}

This conjecture has been established for basis functions on continuous domains \cite{fuselier13}, and we experimentally establish the connection in the next section. 

\section{Examples and simulations}
\label{sec:example}

\subsection{Decay properties}

The Green's function and Lagrange function for a cycle graph are shown in \cref{fig:lagrange_decay_circ}. The graph has 256 nodes, where every fourth vertex is an interpolation node. In \cref{fig:lagrange_decay_lat}, we compare the basis functions on a lattice graph. We see that the Lagrange functions exhibit much faster decay and are much better localized than the corresponding Green's functions. We also see that the local Lagrange is very close to the Lagrange function, especially near the center.

\begin{figure}
    \centering
    \begin{subfigure}[b]{0.32\textwidth}
        \includegraphics[width=\textwidth]{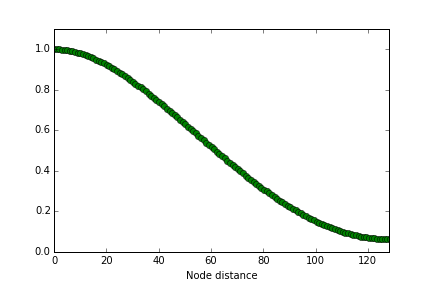}
    \end{subfigure}
    \begin{subfigure}[b]{0.32\textwidth}
        \includegraphics[width=\textwidth]{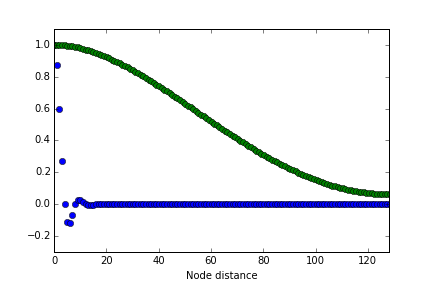}
    \end{subfigure}
    \begin{subfigure}[b]{0.32\textwidth}
        \includegraphics[width=\textwidth]{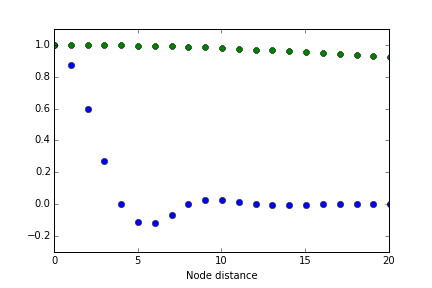}
    \end{subfigure}
    \caption{Comparison of the Green's functions with the Lagrange functions on the cycle graph with 256 nodes. All edges have weight and distance 1. Every fourth vertex is an interpolation node.  \textbf{Left:} Green's function. \textbf{Center:} Greens function vs Lagrange function. \textbf{Right:} Zoomed view of Greens function vs Lagrange function. }
    \label{fig:lagrange_decay_circ}
\end{figure}

\begin{figure}
    \centering
    \begin{subfigure}[b]{0.3\textwidth}
        \includegraphics[width=\textwidth]{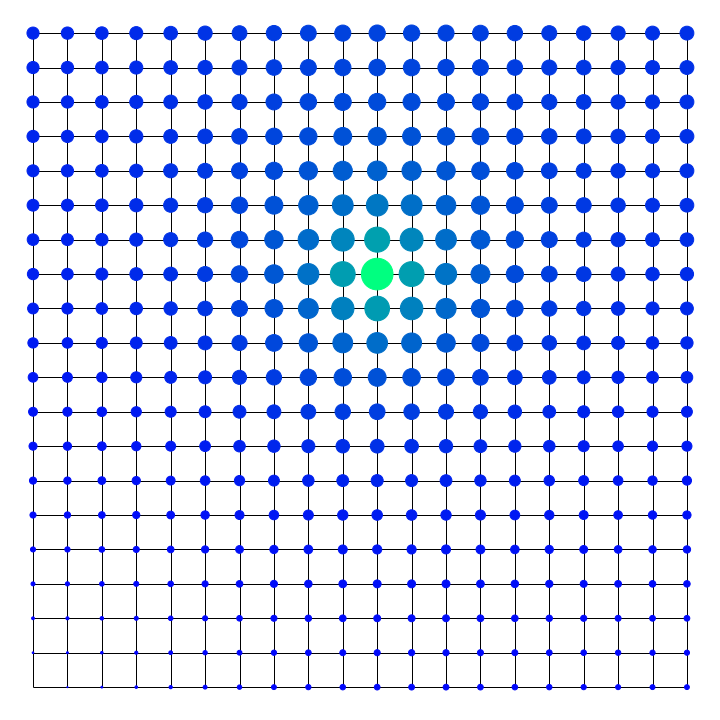}
    \end{subfigure}
    \begin{subfigure}[b]{0.3\textwidth}
        \includegraphics[width=\textwidth]{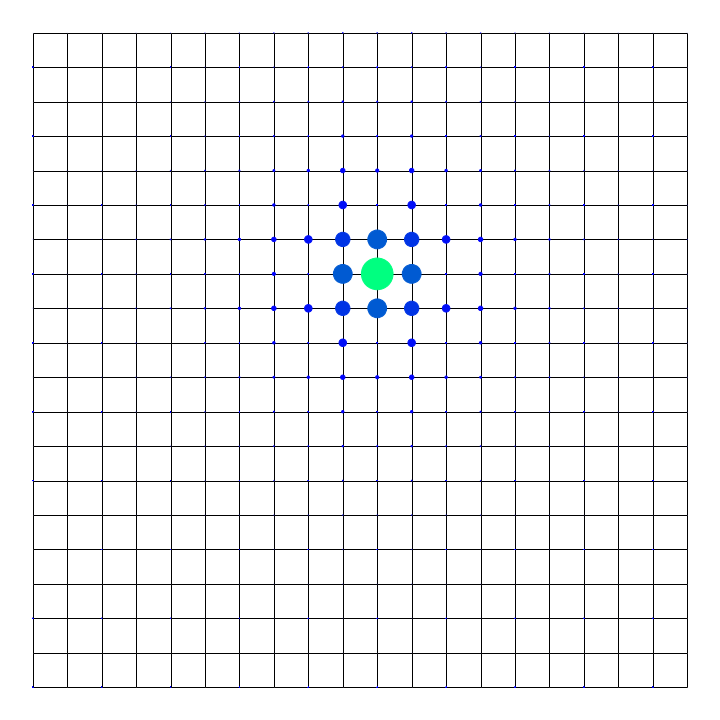}
    \end{subfigure}
    \begin{subfigure}[b]{0.3\textwidth}
        \includegraphics[width=\textwidth]{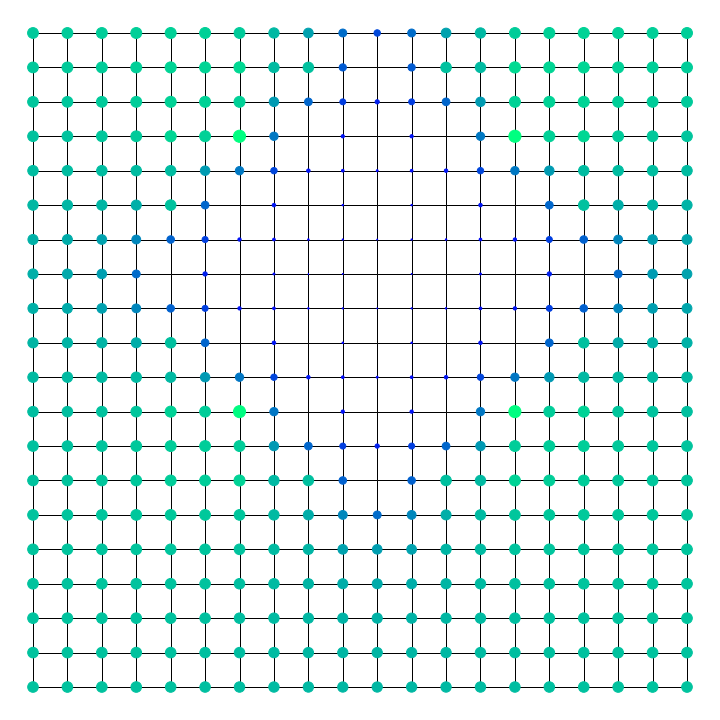}
    \end{subfigure}
    \caption{Comparison of the Green's functions with the Lagrange functions 
and local Lagrange functions 
on the lattice graph with 400 nodes. All edges have weight and distance 1.  The size of the node indicates the absolute value of the function at a vertex. 
\textbf{Left:} Green's function. 
\textbf{Center:} Lagrange function. 
\textbf{Right:} Difference between the Lagrange function and Local Lagrange using interpolation nodes within 6 units of the center. The values are magnified by a factor of 100.
}
    \label{fig:lagrange_decay_lat}
\end{figure}

\subsection{Experiment: Interpolating smooth data}

The purpose of this experiment is to show how the smoothness of the data affects the error of interpolation.
The data was generated as follows. We randomly construct 1000 data sites in the unit square $[0,1]\times[0,1]$, and we created a graph that connects these data sites. Then we define a smooth function that is a sum of uniform translates of a bump function, and we sample it at each data site.  We use half of the data sites as known values and interpolate the other half.

We repeated this experiment with different values for the magnitude of the bump functions, which led to functions of differing smoothness with respect to the graph.  We measured the smoothness of the complete data using $\abs{\cdot}_{H_2^2}$ and compared it with the $\ell_2$-norm of the error of interpolation.  A plot of this comparison is shown in \cref{fig:smoothness_vs_error}. We see a clear trend indicating that the smoothness of the underlying data is closely related to the error of interpolation.

\begin{figure}[ht!]
    \centering
    \includegraphics[width=0.5\textwidth]{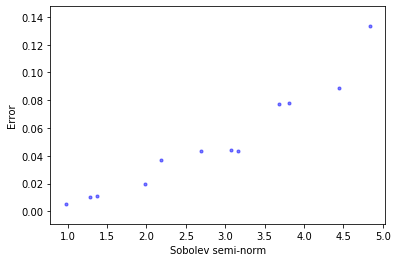}
    \caption{
	Comparison of the smoothness of a given data set and the error of interpolation.      
}
    \label{fig:smoothness_vs_error}
\end{figure}

The bump function that was used is a scaled version of the Wendland function defined by the radial profile 
\begin{equation}
\varphi(r) 
= 
(1-r)^4(4r+1),
\end{equation}
for $r$ in the interval $[0,1]$.

\subsection{Experiment: Comparison with nearest neighbors}

One potential application of our proposed basis functions is kernel-based machine learning on graphs.
The considered problem is regression of energy efficiency data. We predict 
\textit{Heating Load} and 
\textit{Cooling Load} 
based on seven numeric parameters: 
\textit{Relative Compactness}, 
\textit{Surface Area}, 
\textit{Wall Area}, 
\textit{Roof Area}, 
\textit{Overall Height},
\textit{Orientation},
and 
\textit{Glazing Area}. 
The data set comes from the UC Irvine Machine Learning Repository \cite{dheeru17,tsanas12}. 
The data set contains 768 instances.

Our algorithm is the following. 
We normalize each parameter to have mean zero and standard deviation one. 
We construct a graph of all the instances (\textit{known} and \textit{unknown}); 
each instance is a vertex. 
Every vertex is connected to its $K$ nearest neighbors, where $K$ is a connectivity parameter that must be specified by the user.  
The edges are weighted based on the distances between the parameters.  
Then we compute the Lagrange functions at the known data sites. 
Using these basis functions, we interpolate to compute the approximate value for unknown data sites. 
We measure error in terms of mean squared error (MSE) of the approximation.

For the simulation, we split the data set into ten groups for a 10-fold cross-validation. 
For a single test, one of the groups was considered \emph{unknown} while the others were used as \emph{known} data. 
We ran the simulation twenty times, randomizing the data each time. 
The displayed results are an average over these twenty simulations.

We include a comparison with a Nearest Neighbors Regression (NNR) algorithm, \cref{tab:mse}. To make a fair comparison, both algorithms use the same initial clustering of data. After the graph is constructed, the NNR algorithm computes a weighted average of the known neighbors of a given unknown vertex.
We use the same randomized data splits for each algorithm. We ran the experiment using a range of values for the connectivity parameter $K$ in the graph construction. In each case, we see that our method has a lower MSE than NNR.

\begin{table}[ht!]
  \begin{center}
    \caption{Average MSE and standard deviation over twenty simulations with 10-fold cross validation for both \emph{NNR} and our proposed method \emph{Spline}.}
    \label{tab:mse}
\begin{tabular}{c || cc | cc   }
& \multicolumn{2}{c}{Heating Load}   
& \multicolumn{2}{c}{Cooling Load}  \\ 
$K$  & NNR & Spline & NNR & Spline \\ \hline \hline
8 & 5.2319 $\pm$ 0.1778 & 3.7465 $\pm$ 0.1345 & 6.9139 $\pm$ 0.1419 & 5.525 $\pm$ 0.1534\\
10 & 5.2559 $\pm$ 0.1492 & \textbf{3.6792 $\pm$ 0.1071} & 6.7378 $\pm$ 0.1326 & \textbf{5.087 $\pm$ 0.0867}\\
12 & 5.3128 $\pm$ 0.1161 & 3.7629 $\pm$ 0.0837 & 6.801 $\pm$ 0.1355 & 5.2965 $\pm$ 0.1331\\
14 & 5.6949 $\pm$ 0.1033 & 4.2463 $\pm$ 0.0772 & 7.1613 $\pm$ 0.1052 & 5.722 $\pm$ 0.0903\\
16 & 6.266 $\pm$ 0.1099 & 4.9234 $\pm$ 0.0996  & 7.6846 $\pm$ 0.0869 & 6.3822 $\pm$ 0.0926\\
18 & 6.5218 $\pm$ 0.0944 & 5.1405 $\pm$ 0.1031  & 7.9284 $\pm$ 0.1025 & 6.6347 $\pm$ 0.1099\\
\end{tabular}
  \end{center}
\end{table}

\section{Discussion}
\label{sec:discussion}

We have introduced the analog of radial basis functions on graphs and verified properties that are known to hold on continuous domains. Our main result is the decay of the Lagrange functions. Our approach was a spatial domain estimate, rather than the Fourier-based results that are commonly used for uniform data.  This bulk-chasing argument was adapted from the one used on continuous domains for non-uniform data. This connection is encouraging and leads us to believe that additional properties will also carry over. 
A future goal is to show that quasi-interpolating local-Lagrange functions satisfy analogous localization \cite{fuselier13}. 
This is important as it reduces the computational cost of computing the basis functions. The construction of such bases raises questions about how to properly truncate coefficients and the basis functions themselves. Other important properties such as stability with respect to $\ell_p$ are also of interest. Establishing these results should benefit the continuous domain theory as the graph setting is more precise in terms of bounding constants such as $\mu$ of \cref{thm:decay_lagrange} as well as constants appearing in the \textit{footprint radius} of local Lagrange functions.

The potential applications of such bases include kernel-based machine learning algorithms, where data is well represented using a graph framework. Preliminary results are promising, and we expect improved results by adding a refined training and data clustering to our algorithm.

\appendix

\section{Sobolev functions with zeros}
\label{sec:zero_lem}

In this appendix, we consider the unnormalized Laplacian $L = D-A$, where $D$ and $A$ are defined as in \cref{sec:setting}. The Fourier transform of a function is defined in terms of the eigenvalues and eigenvectors of the Laplacian.

\begin{align}
\begin{split}
\widehat{f}(\lambda_n) 
&= 
\innprod{f}{\Lambda_n} \\
f(v)
&=
\sum_{n=0}^{N-1}
\widehat{f}(\lambda_n) \Lambda_n(v)
\end{split}
\end{align}

\begin{customlem}{A.1}\label{lem:zeros}
Let $L$ denote the Laplacian (not normalized) on a graph  $\mathcal{G}$
with $N$ vertices. Let its eigenvectors and eigenvalues be denoted as 
$\set{\Lambda_n}_{n=0}^{N-1}$, 
$\set{\lambda_n}_{n=0}^{N-1}$.
Let $f:\mathcal{V} \rightarrow \mathbb{R}$ 
be a function satisfying $f(v_{0})=0$ for some vertex 
$v_{0} \in \mathcal{V}$. Then
\begin{equation}
	\norm{f}_{\ell_2(\mathcal{G})} 
	\leq
	\frac{\sqrt{N}}{\lambda_1^{\alpha/2}}
	\abs{f}_{H_2^{\alpha}(\mathcal{G})}
\end{equation}
\end{customlem}

\begin{proof}
Our proof is based on the proof of an analogous result in \cite[Chapter 7]{natterer86}. First, we use Parseval's identity to express the left-hand side as a sum of Fourier coefficients
\begin{align}\label{eq:append3}
\begin{split}
	\norm{f}_{\ell_2}^2 
	&=
	\sum_{n=0}^{N-1}
	\abs{\widehat{f}(\lambda_n)}^2\\
	&=
	\abs{\widehat{f}(\lambda_0)}^2 
	+
	\sum_{n=1}^{N-1}
	\abs{\widehat{f}(\lambda_n)}^2.
\end{split}
\end{align}
Since $f$ has a zero at $v_{n_0}$
\begin{equation}\label{eq:append4}
	0
	=
	f(v_{0})
	=
	\widehat{f}(\lambda_0) \Lambda_0(v_{0})
	+
	\sum_{n=1}^{N-1} \widehat{f}(\lambda_n) \Lambda_n(v_{0}).
\end{equation}
The eigenvectors are assumed to have norm 1. In particular, $\Lambda_0$ is the constant vector $1/\sqrt{N}$. Substituting \eqref{eq:append4} into \eqref{eq:append3}, we have
\begin{align}
\begin{split}
	\norm{f}_{\ell_2}^2
	&= 
	\abs{\sqrt{N}\sum_{n=1}^{N-1} 
	\widehat{f}(\lambda_n) \Lambda_n(v_{0})}^2 
	+
	\sum_{n=1}^{N-1}
	\abs{\widehat{f}(\lambda_n)}^2\\ 
	&=
	N\abs{\sum_{n=1}^{N-1} 
	\widehat{f}(\lambda_n) \Lambda_n(v_{0})}^2 
	+
	\sum_{n=1}^{N-1}
	\abs{\widehat{f}(\lambda_n)}^2. 
\end{split}
\end{align}	
Next, we apply the Cauchy-Schwartz inequality
\begin{align}
\begin{split}
	\norm{f}_{\ell_2}^2
	&\leq
	N
	\parenth{\sum_{n=1}^{N-1} \abs{\Lambda_n(v_{0})}^2}
	\parenth{\sum_{n=1}^{N-1} \abs{\widehat{f}(\lambda_n)}^2}
	+
	\sum_{n=1}^{N-1}
	\abs{\widehat{f}(\lambda_n)}^2 \\
	&=
	N
	\parenth{1-\frac{1}{N}}
	\parenth{\sum_{n=1}^{N-1} \abs{\widehat{f}(\lambda_n)}^2}
	+
	\sum_{n=1}^{N-1}
	\abs{\widehat{f}(\lambda_n)}^2 \\
	&=
	\parenth{N-1}
	\parenth{\sum_{n=1}^{N-1} \abs{\widehat{f}(\lambda_n)}^2}
	+
	\sum_{n=1}^{N-1}
	\abs{\widehat{f}(\lambda_n)}^2 \\
	&=
	N
	\sum_{n=1}^{N-1} \abs{\widehat{f}(\lambda_n)}^2
\end{split}
\end{align}
Finally we have
\begin{align}
\begin{split}
	\norm{f}_{\ell_2}^2
	&\leq
	\frac{N}{\lambda_1^{\alpha}}
	\sum_{n=1}^{N-1}
	\abs{\lambda_n^{\alpha/2}
	\widehat{f}(\lambda_n)}^2\\ 
	&=
	\frac{N}{\lambda_1^{\alpha}}
	\abs{f}_{H_2^{\alpha}}^2
\end{split}
\end{align}
\end{proof}

\bibliographystyle{plain}
\bibliography{interp_arxiv}

\end{document}